\documentclass[11pt,reqno,tbtags,a4paper]{amsart}
\usepackage{amssymb}
\usepackage{xpunctuate}
\usepackage{url}
\usepackage[square,numbers]{natbib}
\bibpunct[, ]{[}{]}{;}{n}{,}{,}

\title[smooth graphons in different dimensions]
{Can smooth graphons in several dimensions 
be represented by smooth graphons on $[0,1]$?}

\date{19 January, 2021}

\author{Svante Janson}
\thanks{Supported by the Knut and Alice Wallenberg Foundation}
\address{Department of Mathematics, Uppsala University, PO Box 480,
SE-751~06 Uppsala, Sweden}
\email{svante.janson@math.uu.se}
\newcommand\urladdrx[1]{{\urladdr{\def~{{\tiny$\sim$}}#1}}}
\urladdrx{http://www2.math.uu.se/~svante/}

\author{Sofia Olhede}
\thanks{Supported by the European Research Council under Grant CoG2015-682172NETS}
\address{Institute of Mathematics, Ecole Polytechnique F\'ed\'erale de Lausanne,Lausanne, Switzerland}
\email{sofia.olhede@epfl.ch}


\numberwithin{equation}{section}

\renewcommand\le{\leqslant}
\renewcommand\ge{\geqslant}

\allowdisplaybreaks

\theoremstyle{plain}
\newtheorem{theorem}{Theorem}[section]
\newtheorem{lemma}[theorem]{Lemma}
\newtheorem{proposition}[theorem]{Proposition}

\theoremstyle{definition}

\newtheorem{exampleqqq}[theorem]{Example}
\newenvironment{example}{\begin{exampleqqq}}
  {\hfill\qedsymbol\end{exampleqqq}}

\newtheorem{remarkqqq}[theorem]{Remark}
\newenvironment{remark}{\begin{remarkqqq}}
  {\hfill\qedsymbol\end{remarkqqq}}

\theoremstyle{remark}

\newenvironment{romenumerate}[1][-10pt]{
\addtolength{\leftmargini}{#1}\begin{enumerate}
 \renewcommand{\labelenumi}{\textup{(\roman{enumi})}}%
 }{\end{enumerate}}

\newcounter{oldenumi}
{\setcounter{oldenumi}{\value{enumi}}
\begin{romenumerate} \setcounter{enumi}{\value{oldenumi}}}
{\end{romenumerate}}

\newcounter{thmenumerate}

\newcounter{xenumerate}   

\newcommand{\refT}[1]{Theorem~\ref{#1}}

\newcommand{\refL}[1]{Lemma~\ref{#1}}
\newcommand{\refLs}[1]{Lemmas~\ref{#1}}

\newcommand{\refS}[1]{Section~\ref{#1}}

\newcommand{\refP}[1]{Proposition~\ref{#1}}

\newcommand{\refE}[1]{Example~\ref{#1}}

\newcommand\marginal[1]{\marginpar[\raggedleft\tiny #1]{\raggedright\tiny#1}}
\newcommand\REM[1]{{\raggedright\texttt{[#1]}\par\marginal{XXX}}}
\newcommand\XREM[1]{\relax}

\begingroup
  \divide\count255 by 60
  \multiply\count255 by -60
  \advance\count255 by \time
  \ifnum \count255 < 10 \xdef\klockan{\the\count1.0\the\count255}
  \else\xdef\klockan{\the\count1.\the\count255}\fi
\endgroup

\newcommand{\sumko}{\sum_{k=0}^\infty}

\newcommand{\sumid}{\sum_{i=1}^d}

\newcommand\set[1]{\ensuremath{\{#1\}}}

\newcommand\bigpar[1]{\bigl(#1\bigr)}
\newcommand\Bigpar[1]{\Bigl(#1\Bigr)}

\newcommand\lrpar[1]{\left(#1\right)}

\newcommand\bigabs[1]{\bigl\lvert#1\bigr\rvert}

\newcommand\lrabs[1]{\left\lvert#1\right\rvert}
\def\rompar(#1){\textup(#1\textup)}    

\def\xexp(#1){e^{#1}}

\newcommand\norm[1]{\lVert#1\rVert}

\newcommand\punkt{\xperiod}    

\newcommand\eg{e.g\punkt}

\newcommand\bbR{\mathbb R}

\newcommand\bbT{\mathbb T}

\newcounter{CC}
\newcounter{cc}

\newcommand\ga{\alpha}
\newcommand\gb{\beta}

\newcommand\gf{\varphi}

\renewcommand\phi{\xxx}  

\newcommand\cF{\mathcal F}

\newcommand\cS{{\mathcal S}}

\newcommand\intoi{\int_0^1}

\newcommand\oi{\ensuremath{[0,1]}}

\newcommand\dd{\,\mathrm{d}}

\newcommand\intS{\int_{\cS}}
\newcommand\intoid{\int_{\oi^d}}
\newcommand\Holderx[1]{\text{\Holder$(#1)$}}

\newcommand\dox[1]{|#1|_\circ}

\newcommand{\Holder}{H\"older}

\begin{document}

\begin{abstract} 
A graphon that is defined on $[0,1]^d$ and is H\"older$(\alpha)$ continuous
for some $d\ge2$ and $\alpha\in(0,1]$
can be represented by a graphon on $[0,1]$ that is 
H\"older$(\alpha/d)$ continuous.
We give examples that show that this reduction in smoothness to $\alpha/d$
is the best possible, for any $d$ and $\alpha$; 
for $\alpha=1$, the example is a dot product graphon and 
shows that the reduction is the best possible
even for graphons that are  polynomials.

A motivation for studying
the smoothness of graphon functions is that this represents a key assumption
in non-parametric statistical network analysis. 
Our examples show that
making a smoothness assumption in a particular dimension is not equivalent
to making it in any other latent dimension.
\end{abstract}

\maketitle

\section{Introduction}\label{S:intro}
Networks or graphs are a convenient and parsimonious data structure for
representing
objects and their interactions. Initial interest in networks in statistics
has focussed on fitting 
simple and parametric models to summarize data structure~\cite{Kolaczyk},
such as the Chung-Lu or expected degree model, or variants 
of the stochastic block model. What most statistical network models satisfy
is a probabilistic invariance to permutations, and 
this invariance leads to a natural representation of a graph generating
mechanism via a graphon or a graph limit function~\cite{Lovasz} via the
Aldous--Hoover theorem.

In general, a graphon can be defined on any probability space 
$\cS=(\cS,\cF,\mu)$.
A \emph{graphon} on $\cS$ is a symmetric measurable function
$W:\cS^2\to\oi$.
As is well known, graphons representing a graph limit or a random graph
model
are not unique, and there is the notion of (weak)
equivalence of graphons; see \cite{Lovasz}. 
In particular, any graphon is equivalent to a graphon
defined on $\oi$; thus from an abstract point of view, it 
suffices to consider this case, and indeed, several papers consider only
such graphons.
However, in applications, it is often useful to consider other spaces $\cS$,
since not all models of networks are naturally formulated in terms of a
graphon on $[0,1]$.
In particular, 
it is often natural to use
subsets of $\bbR^d$ with $d\ge2$;
some examples are
the random dot product model~\cite{Athreya}, 
and 
applications where the latent dimension is interpreted as a position in a
social space, cf.~\cite{Hoff}.
We consider below the case $\cS=\oi^d$ (with Lebesgue measure);
this means that each node is assigned $d$ latent variables, which are
independent and uniformly distributed on $\oi$.

From a statistical perspective,
thus at best we
can only estimate an element of the equivalence class of a graphon, 
just like in
statistical shape analysis, where we may estimate a shape but we have to
factor out shifts and rotations, as they do not alter the underlying
shape~\cite{Chikuse}. 
The probabilistic invariance that is most closely studied in statistics is
an invariance to temporal and spatial shifts, most commonly found in
stochastic processes~\cite{Adler}. 
The graphon function by analogy can therefore be compared to the spectral
density of a stochastic process, except it does not permit as easy a
characterisation as the spectral density of a random field or time
series. Despite this fact, to enable estimation in statistics assumptions of
regularity of a graphon, such as Lipschitz or \Holder{} continuity, 
has become common when analysing networks
non-parametrically~\cite{Gao,Olhede,Orbanz}.  

Despite the recent progress in statistics, machine learning and network data
analysis, it is unclear how restrictive the assumption of either a Lipschitz
or H\"older$(\alpha)$ graphon on $[0,1]$ is, and also
how it compares with such assumptions for graphons defined on other 
probability spaces $\cS$.

The aim of this paper is to explore
the consequences of assuming smoothness of a graphon when
its arguments takes values in $[0,1]^d$. 
It is easy to see that any $\Holderx{\ga}$ graphon on $[0,1]^d$ ($d\ge2$) is
equivalent to a H\"older($\ga/d$) graphon on $[0,1]$;
in particular any Lipschitz smooth graphon on $[0,1]^d$ is
equivalent to a H\"older($1/d$) graphon on $[0,1]$  (\refT{T1}).
Moreover, we give examples showing that in general this is the best possible.
In particular, we exhibit a simple infinitely differentiable graphon on
$\oi^d$ that is not equivalent to any
H\"older($q$) graphon on $[0,1]$ for $q>1/d$. 
The interpretation of this is
that a smoothness assumption in a particular dimension ``has teeth'' and
thus represents a real restriction, which furthermore depends on the dimension.

What is the statistical importance of that result? By
assuming smoothness we are able to exhibit a member of the equivalence class
of graphon functions and so bound any approximation error going from a block model
to a H\"older($\alpha$) smooth function, drawing on classical results in
numerical analysis~\cite{DeVore} and the convergence of order statistics,
see e.g.~\cite{Olhede}.
Furthermore by
averaging we reduce variance and so make the average block heights nicely
behaved random variables (controlled tail behaviour), irrespectively of what
groupings we keep in a block.  So the urge to average is natural, as
so many nice results come from this act. However it comes at a price, namely to justify averaging in blocks we need to assume graphon smoothness 
in $[0,1]$ and not all graphons will satisfy this assumption.

\section{Notation and main results}

Recall that, for a given  $\ga\in(0,1]$,
 a function $f$
defined on a subset $\cS$ of a Euclidean space $\bbR^d$, say,
is \emph{\Holderx{\ga}} if there exists a constant $C<\infty$ such that
\begin{align}\label{holder}
  |f(x)-f(y)|\le C|x-y|^\ga,
\qquad x,y\in\cS.
\end{align}
Functions that are \Holderx1{} are also called \emph{Lipschitz}.
In particular, this notion applies to graphons $W$ defined on $\oi^d$; recall
that then $W$ is a function on $\oi^{2d}$.

As said above, graphons are not unique, 
see \eg{} \cite{Lovasz} and \cite{SJ249}.
In particular, if $\cS_1$ and $\cS_2$ are two probability spaces and
$\gf:\cS_1\to\cS_2$ is a measure-preserving map, then for any graphon $W$
on $\cS_2$, its \emph{pull-back} $W^\gf(x,y):=W\bigpar{\gf(x),\gf(y)}$
is a graphon on $\cS_1$ that is equivalent to $W$.
The converse does not hold, but it holds ``almost'', see \refP{P0} below.

We note first a simple result showing that 
every \Holder{} continuous graphon on $\oi^d$
is equivalent to a
graphon on $\oi$ that is 
\Holder{} continuous albeit with a different \Holder{} exponent after the
change of dimension.
We regard $\oi$ and $\oi^d$ as probability spaces equipped
with the usual  Lebesgue measure.

\begin{theorem}\label{T1}
  Let\/ $W$ be a graphon on $\oi^d$ that is $\Holderx\ga$ 
for some   $d\ge2$ and $\ga\in(0,1]$.
Then there exists an equivalent graphon on $\oi$ that is \Holderx{\ga/d}.

In particular, if\/ $W$ is differentiable, or just Lipschitz, then
there exists an equivalent graphon on $\oi$ that is \Holderx{1/d}.
\end{theorem}

\begin{proof}
  Several standard constructions of Peano curves yield a measure-pre\-serving
  map $\gf:\oi\to\oi^d$ that is \Holderx{1/d},
see \eg{} \cite{Sagan} and \cite{deFreitas}. 
Then the pull-back $W^\gf$ is a graphon on $\oi$ that
is equivalent to $W$ and is \Holderx{1/d}.
\end{proof}

Our main purpose is to show that \refT{T1} is the best possible, by exhibiting
 graphons, for which the exponent $\ga/d$ cannot be improved.

\begin{example}\label{E1}
  Let $W$ be the graphon associated with the random dot product graph on $\oi^d$ given by
  \begin{align}\label{WE1}
    W(x,y) = a x\cdot y = c(x_1y_1+\dotsm+ x_dy_d),
\qquad x,y\in\oi^d,
  \end{align}
where $a>0$ is a constant, $\cdot$ is the scalar product, and 
$x=(x_1,\dots,x_d)$,
$y=(y_1,\dots,y_d)$.
(The constant $a$ may be chosen as $1/d$ to make $0\le W\le 1$.)

$W$ is a polynomial and thus infinitely differentiable.
We show in \refS{Spf} that $W$ is not equivalent to any graphon on $\oi$
that is \Holderx{\ga} for any $\ga>1/d$; in particular not to any 
Lipschitz or differentiable graphon on $\oi$.
\end{example}

\begin{example}\label{E2}
Let $d\ge 2$ and $\ga\in(0,1)$.
Let $h_\ga(t)$ be the 
Weierstrass function
given by the lacunary Fourier series
\begin{align}
  h_\ga(t):=\sumko 2^{-k\ga} \cos\bigpar{2\pi 2^k t}.
\end{align}
Then $h_\ga$ is a
symmetric and periodic real-valued function on $\bbR$. 
Furthermore, it is easy to see that
$h_\ga\in\Holderx\ga$, see \cite[Theorem II.(4.9)]{Zygmund}.

Now define, for
$x=(x_1,\dots,x_d)$ and
$y=(y_1,\dots,y_d)$ in $\oi^d$,
\begin{align}\label{WE2}
  W(x,y):=\frac12+a\sumid h_\ga(x_i-y_i),
\end{align}
where $a>0$ is chosen so small that $0\le W(x,y)\le 1$.
Then $W$ is a graphon on $\oi^d$, and $W$ is $\Holderx\ga$ since $h_\ga$ is.
By \refT{T1},
there exists a graphon $W'$ on
$\oi$ that is equivalent to $W$ and which is $\Holderx{\ga/d}$.
We show in \refS{Spf} that this is the best possible;
$W$ is not equivalent to any graphon on $\oi$
that is \Holderx{\gb} for any $\gb>\ga/d$.
\end{example}



\section{Proofs}\label{Spf}

We first quote
the following characterization of equivalence of
graphons, proved by
\citet{BCL:unique}.
(See also \cite[Theorem 13.10]{Lovasz} and \cite[Theorems 8.3 and 8.4]{SJ249}.)
\begin{proposition}[\citet{BCL:unique}]\label{P0}
 Two graphons $W_1$ and $W_2$, 
defined on probability spaces $\cS_1$ and $\cS_2$, respectively, 
are equivalent if and only if there exists a third graphon $W$ on a
probability space $\cS$ and two measure-preserving maps $\gf_1:\cS_1\to\cS$ 
and $\gf_2:\cS_2\to\cS$ such that $W_j$ a.e.\ equals the pull-back
$W^{\gf_j}$,
$j=1,2$. 
\end{proposition}

Our proofs are based on  a  functional 
of graphons, defined as follows.
Let $q>0$.
For a graphon $W$ on a probability space $\cS$, 
or more generally any  measurable function $W:\cS^2\to\bbR$,
we define 
\begin{align}
  \label{Psi}
\Psi_q(W) :=
  \intS  \intS 
\lrpar{ \intS|W(x,z) - W(y,z)| \dd\mu(z)}^{-q} 
\dd\mu(x)\dd\mu(y) 
\le\infty
.\end{align}
This functional is related to integrals used for, e.g., $L^p$-versions of
\Holder{} continuity, but note the negative power; thus $\Psi_q(W)$ is large
(or infinite) when $W$ is sufficiently smooth, and $1/\Psi_q(W)$ may be
regarded as a special  kind of  measure of (lack of) smoothness.

The claims in the examples will follow from the lemmas below.

\begin{lemma}\label{L1}
If\/ $W$ and $W'$ are two equivalent graphons, possibly defined on different
  probability spaces, then
$\Psi_q(W)=\Psi_q(W')$ for every $q>0$.
\end{lemma}

\begin{proof}
  By \refP{P0},
it suffices to prove this when $W'$ is a.e.\ equal to
a pull-back of $W$ by a measure-preserving map.
This case follows by trivial changes of variables in the integrals.
\end{proof}

\begin{lemma}\label{L2}
If\/ $\ga>0$ and\/ $W$ is a graphon on $\oi$ such that 
$W$ is \Holderx{\ga}, 
then
$\Psi_q(W)=\infty$ for every $q\ge 1/\ga$.
\end{lemma}

\begin{proof}
  By assumption,
$|W(x,z)-W(y,z)|\le C|x-y|^\ga$,
and thus
$\intoi|W(x,z)-W(y,z)|\dd z\le C|x-y|^\ga$. Hence,
 \eqref{Psi} yields
\begin{align}
  \Psi_q(W) \ge C^{-q}\intoi\intoi |x-y|^{-q\ga}\dd x\dd y
=\infty,
\end{align}
since $q\ga\ge1$.
\end{proof}

\begin{lemma}\label{L3}
  If\/ $d\ge1$, $\cS=\oi^d$ 
and $W(x,y):=a x\cdot y$ for $x,y\in\oi^d$ and some $a>0$,
then  $\Psi_q(W)<\infty$ for every $q< d$.
\end{lemma}
\begin{proof}
  By homogeneity, we may without loss of generality assume $a=1$.
Then 
\begin{align}\label{ca}
|W(x,z)-W(y,z)|=|x\cdot z - y \cdot z| =|(x-y)\cdot z|.
\end{align}
Define, for $x\in \bbR^d$,
\begin{align}\label{cb}
  h(x):=\intoid|x\cdot z|\dd z.
\end{align}
Then $h(x)$ is a continuous function of $x$, and $h(x)>0$ for $x\neq0$.
Hence, $c_d:=\inf \set{h(x):|x|=1}>0$ by compactness of the unit sphere.
Furthermore, homogeneity yields $h(x)\ge c_d |x|$ for every $x\in \bbR^d$.
Consequently, using \eqref{ca},
\begin{align}\label{cc}
\intoid|W(x,z)-W(y,z)|\dd z
= h(|x-y|)
\ge c_d |x-y|,
\end{align}
and the definition \eqref{Psi} yields
\begin{align}\label{cd}
  \Psi_q(W) \le c_d^{-q} \intoid\intoid |x-y|^{-q}\dd x\dd y
<\infty,
\end{align}
recalling the assumption $q<d$.
\end{proof}

\begin{proof}[Proof of claim in \refE{E1}]
  Suppose that $W'$ is a graphon on $\oi$ that is equivalent to $W$ and
also is $\Holderx\ga$ for some $\ga>1/d$. Take $q:=1/\ga<d$.
Then $\Psi_q(W')=\infty$ by \refL{L2} and $\Psi_q(W)<\infty$ by \refL{L3},
which contradicts \refL{L1}.
\end{proof}

\begin{lemma}
  \label{LE2}
If\/ $d\ge2$, $0<\ga<1$ and\/ $W$ is given by \eqref{WE2},
then $\Psi_q(W)<\infty$ for every $q<d/\ga$.
\end{lemma}

\begin{proof}
 Define for $x,y\in\oi$,
\begin{align}
  \dox{x-y}:=\min\bigpar{|x-y|,1-|x-y|}
\end{align}
and define, more generally, for 
$x=(x_i)_1^d\in\oi^d$ and $y=(y_i)_1^d\in\oi^d$,
\begin{align}
\dox{x-y}:=\sum_{i=1}^d\dox{x_i-y_i}.  
\end{align}
(These can be regarded as metrics on $\bbT$ and $\bbT^d$, where the unit
circle $\bbT$  is regarded as
$\oi$ with the endpoints 0 and 1 identified.)

The function $h_\ga$ satisfies
for some $c_1,c_2>0$ and all $y\in(0,1)$,
\begin{align}\label{h1}
  \norm{h(\cdot)-h(\cdot-y)}_{L^1\oi}
\ge
c_1  \norm{h(\cdot)-h(\cdot-y)}_{L^2\oi}
\ge c_2 \dox{y}^\ga
,\end{align}
where the first inequality is a general property of lacunary series
\cite[Theorem V.(8.20)]{Zygmund}
and the second follows by Parseval's relation and a simple calculation which
we omit.
Consequently, \eqref{WE2} yields, using \eqref{h1} in the last line,
for any $y=(y_i)_1^d\in\oi^d$, 
\begin{align}\qquad&\hskip-2em
  \int_{\oi^d}\bigabs{W(x,z)-W(y,z)}\dd z
=a\int_{\oi^d}\lrabs{\sumid \Bigpar{h(x_i-z_i)-h(y_i-z_i)}}\dd z
\notag\\&
\ge a\intoi\lrabs{\int_{\oi^{d-1}}\sumid\Bigpar{ h(x_i-z_i)-h(y_i-z_i)}\dd
  z_2\dotsm\dd z_d}\dd z_1
\notag\\&
= a\intoi\bigabs{h(x_1-z_1)-h(y_1-z_1)}\dd z_1
=a\norm{h(\cdot)-h(\cdot-x_1+y_1)}_{L^1\oi}
\notag\\&
\ge a c_2 \dox{x_1-y_1}^\ga.
\end{align}
By symmetry, we also have the lower bound $a c \dox{x_i-y_i}^\ga$ for any
$i\le d$, and thus
\begin{align}\label{h6}
    \int_{\oi^d}\bigabs{W(x,z)-W(y,z)}\dd z
\ge ac_2\frac{1}{d}\sumid\dox{x_i-y_i}^\ga\ge \frac{ac_2}d\dox{x-y}^\ga.
\end{align}
The estimate \eqref{h6} implies that
$\Psi_q(W)<\infty$ for every 
$q<d/\ga$, similarly to \eqref{cd}.
\end{proof}

\begin{proof}[Proof of claim in \refE{E2}]
As for \refE{E1}, now using \refLs{L1}, \ref{L2} and \ref{LE2}.
\end{proof}

\begin{remark}
  Although the proofs are for specific examples, the arguments suggest that
the conclusions of 
\refL{L3} and \ref{LE2}, and as a consequence the conclusions of
\refE{E1} and \refE{E2}, are typical of graphons on
$\oi^d$ with the given smoothness, rather than exceptional.
\end{remark}

\newcommand\AMS{Amer. Math. Soc.}
\newcommand\Springer{Springer-Verlag}
\newcommand\Wiley{Wiley}

\newcommand\vol{\textbf}
\newcommand\jour{\emph}
\newcommand\book{\emph}
\newcommand\inbook{\emph}
\def\no#1#2,{\unskip#2, no. #1,} 
\newcommand\toappear{\unskip, to appear}

\newcommand\arxiv[1]{\texttt{arXiv}:#1}
\newcommand\arXiv{\arxiv}

\def\nobibitem#1\par{}

\end{document}